\documentclass[12 pt, reqno]{amsart}
\usepackage[all]{xy}

\usepackage[urlcolor=blue]{hyperref} % for clickable toc and references

\usepackage{enumitem} % for parsep, itemsep...
\usepackage{xspace} % for \xspace

\usepackage{amsfonts}
\usepackage{amsmath,amssymb}      % Better maths support & $more symbols
\usepackage{amsthm}
\usepackage{pdfsync}       % enable tex source and pdf output synchronicity
\usepackage{color}

\usepackage[T1]{fontenc}

\usepackage{color}
\usepackage{float}
\usepackage{tikz,subfigure}

\usetikzlibrary{arrows}
\tikzset{
v/.style={
  circle, draw, inner sep=2pt, minimum size=6pt, fill=white},
l/.style={
  circle, draw, inner sep=2pt, minimum size=6pt, fill=black}
}

\theoremstyle{plain}
\newtheorem{theorem}{Theorem}[section]
\newtheorem{proposition}[theorem]{Proposition}
\newtheorem{lemma}[theorem]{Lemma}
\newtheorem{corollary}[theorem]{Corollary}

\theoremstyle{definition}
\newtheorem{example}[theorem]{Example}
\newtheorem{definition}[theorem]{Definition}
\newtheorem{remark}[theorem]{Remark}

 \newcommand \C {{\mathbb C}}
 \newcommand \A{{\mathcal A}}
  \newcommand \B{{\mathcal B}}
  \newcommand \G{{\mathcal G}}

\DeclareMathOperator{\codim}{codim}
\newcommand \spann  {{\rm span}}

\renewcommand\ell{l}

\def \X(#1){\{x_1,\dots, x_{#1}\}}

\usepackage[utf8]{inputenc}
\usepackage[english]{babel}
\usepackage{fancyhdr}
 
\begin{document}

\title {On the Falk invariant of hyperplane arrangements attached to gain graphs}
\begin{abstract} The fundamental group of the complement of a hyperplane arrangement in a complex vector space is an important topological invariant. The third rank of successive quotients in the lower central series of the fundamental group was called the \emph{Falk invariant} of the arrangement since Falk gave the first formula and asked for a combinatorial interpretation. In this article, we give a combinatorial formula for the Falk invariant of hyperplane arrangements attached to certain gain graphs.
\end{abstract}

\author{Weili Guo}
\address{Department of Mathematics,
Beijing University of Chemical Technology, Beijing,100013, China}
\email{guowl@mail.buct.edu.cn}
\author{Michele Torielli}
\address{Department of Mathematics, GI-CoRE GSB, Hokkaido University, Sapporo 060-0810, Japan.}
\email{torielli@math.sci.hokudai.ac.jp}

\date{\today}
\maketitle

%\tableofcontents\textbf

\section{Introduction}
A \textbf{hyperplane} $H$ in $\C^\ell$ is an affine subspace of dimension $\ell-1$. A finite collection $\A = \{H_1, \dots , H_n\}$ of hyperplanes is called a \textbf{hyperplane arrangement}. If $\bigcap_{i=1}^{n}H_i\ne\emptyset$, then $\A$ is called \textbf{central}. In this paper, we only consider central arrangements and assume that all the hyperplanes contain the origin. For more details on hyperplane arrangements, see \cite{orlterao}.

Let $M:=\C^\ell\setminus\bigcup_{H\in\A}H$ be the complement of the arrangement $\A$. It is known that the cohomology ring $H^*(M)$ is completely determined by $L(\A)$, the lattice of intersections of $\A$. There are several conjectures concerning the relationship between $M$ and $L(\A)$. To study such problems, Falk introduced in \cite{falk1990algebra} a multiplicative invariant, called \textbf{global invariant}, of the Orlik--Solomon algebra of $\A$. The invariant is now known as the ($3^{rd}$) \textbf{Falk invariant} and it is denoted by $\phi_3$. In \cite{falk2001combinatorial}, Falk posed as an open problem the finding of a combinatorial interpretation of $\phi_3$. %This article is devoted to give a partial answer to this problem for the case of signed graphic arrangements. 

Several authors already studied this invariant. In \cite{schenck2002lower}, Schenck and Suciu studied the lower central series of arrangements and described a formula for the Falk invariant in the case of graphic arrangements. In \cite{guo2017global}, the authors gave a formula for $\phi_3$ in the case of signed graphic arrangements, whose underlying  signed graph has only simple edges and no loops. In \cite{guo2017falkinvar}, the authors extended the previous result to signed graphic arrangements coming from graphs without loops. In \cite{guo2017falk}, we described a combinatorial formula for the Falk invariant of several signed graphic arrangements with loops. In this paper, we will describe a combinatorial formula for the Falk invariant $\phi_3$ for $\mathcal{A}(\G)$, an arrangement associated to certain gain graphs. % $\G$ that does not have a subgraph isomorphic to $\pm 1K_2^\circ$, it has no loops adjacent to a theta graph with only three edges and it has at most triple parallel edges. 
Since a signed graph is a special case of the type of gain graphs considered in this paper, the previous results will follow.

The paper is organized as follows. In Section 2, we recall the notions of Orlik--Solomon algebras and the Falk invariant. In Section 3, we recall the definitions and basic properties of  gain graphs. In Section 4, we list all the gain graphs that will play a role in our main theorem. In Section 5, we state and prove our main theorem. In Section 6, we give a matroidal interpretation of our main theorem.

All the computations in this article have been performed using the computer algebra software CoCoA, see \cite{palezzato2019hyperplane}.

\section{Preliminares on Orlik--Solomon algebras}
Let $\A=\{H_1, \dots, H_n\}$ be a central arrangement of hyperplanes in $\C^{\ell}$.
Let $E^1:=\bigoplus_{j=1}^n\C e_j$ be the free module generated by $e_1, e_2, \dots, e_n$, where $e_i$ is a symbol corresponding to the hyperplane $H_i$.
Let $E:=\bigwedge E^1$ be the exterior algebra over $\C$. The algebra $E$ is graded via $E=\bigoplus_{p=0}^nE^p$, where $E^p:=\bigwedge^pE^1$.
The $\C$-module $E^p$ is free and has the distinguished basis consisting of monomials $e_S:=e_{i_1}\wedge\cdots\wedge e_{i_p}$,
where $S=\{{i_1},\dots, {i_p}\}$ is running through all the subsets of $\{1,\dots,n\}$ of cardinality $p$ with $i_1<i_2<\cdots<i_p$.
The graded algebra $E$ is a commutative differential graded algebra with respect to the differential $\partial$ of degree $-1$ uniquely defined
by the conditions $\partial e_i=1$ for all $i=1,\dots, n$ and the graded Leibniz formula. Then for every $S\subseteq\{1,\dots,n\}$ of cardinality $p$, we have
$$\partial e_S=\sum_{j=1}^p(-1)^{j-1}e_{S_j},$$
where $S_j$ is the complement in $S$ to its $j$-th element.

For $S\subseteq\{1,\dots,n\}$, put $\bigcap S:=\bigcap_{i\in S}H_i$. The set of all
intersections $L(\A):=\{\bigcap S\mid S\subseteq \{1,\dots,n\}\}$ is called the \textbf{intersection lattice of $\A$}.
A subset $S\subseteq\{1,\dots,n\}$ is called \textbf{dependent} if the set of linear polynomials $\{\alpha_i~|~i\in S\}$, with $H_i=\alpha_i^{-1}(0)$, is linearly dependent. This definition a priori depends on the choice of defining linear forms $\alpha_i$, however it depends only on the hyperplanes in $S$. In fact, it is equivalent to ask that $|S| > \codim(S)$.
\begin{definition}\label{def:osalgbr}
The \textbf{Orlik--Solomon ideal} of $\A$ is the ideal $I=I(\A)$ of $E$ generated by $\{\partial e_S~|~S \text{ dependent }\}$.
%\begin{enumerate}
%\item[$(1)$] all $e_S$ with $\bigcap S=\emptyset$,
%\item[$(2)$] all $\partial e_S$ with $S$ dependent.
%\end{enumerate}
The algebra $A:=A^\bullet(\A)=E/I(\A)$ is called the \textbf{Orlik--Solomon algebra} of $\A$.
\end{definition}
Clearly $I$ is a homogeneous ideal of $E$ and $I^p=I\cap E^p$ whence $A$ is a graded algebra and we can write $A=\bigoplus_{p\ge 0} A^p$, where $A^p=E^p/I^p$.
%If $\A$ is central, then for any $S\subseteq\A$, we have $\bigcap S\neq\emptyset$. Therefore, the Orlik--Solomon ideal is generated only
%by the elements of type $(2)$ from Definition \ref{def:osalgbr}.
The map $\partial$ induces a well-defined differential $\partial\colon A^p(\A)\longrightarrow A^{p -1}(\A)$, for any $p>0$.

Let $I_k$ be the ideal of $E$ generated by $\sum_{j\le k}I^j$. We call $I_k$ the
 \textbf{$k$-adic Orlik--Solomon ideal} of $\A$. It is clear that $I_k$ is a graded ideal and $(I_k)^p = E^p\cap I_k$. Write $A_k:= E/I_k$ and
$A_k^p:= E^p/(I_k)^p$ which is called \textbf{$k$-adic Orlik--Solomon algebra} by Falk \cite{falk1990algebra}.

%\section{Falk invariant $\phi_3$}
%Let $M$ be the complement of a hyperplane arrangement $\A$ in $\C^\ell$. The fundamental group $\pi(M)$ of $M$ is an important and complicated invariant of $\A$. The lower central series of $\pi(M)$ is a chain of normal subgroups $N_1 = \pi(M)$, and for $k \ge 1$, $N_{k+1} = [N_k,N_1]$, the subgroup generated by commutators of elements in $N_k$ and $N_1$. The rank $\phi_3$ of the finitely generated abelian group $N_3/N_4$ is called \textbf{Falk invariant of $\A$} since Falk \cite{falk1997arrangements}, \cite{falk2001combinatorial} and \cite{falk1990algebra}.

In this setup, it is now easy to define the Falk invariant.
\begin{definition} Consider the map $d$ defined by
$$d\colon E^1\otimes I^2\to E^3,$$
$$d(a\otimes b)=a\wedge b.$$
Then the \textbf{Falk invariant} is defined as
$$\phi_3:=\dim(\ker(d)).$$
\end{definition}

In \cite{falk1990algebra} and \cite{falk2001combinatorial}, Falk gave a beautiful formula to compute such invariant. In \cite{falk2001combinatorial}, there is a typo in the formula, the correct one is the one described below.

\begin{theorem}\cite[Theorem 4.7]{falk2001combinatorial}\label{theo:falkinvar} Let $\A=\{H_1, \dots, H_n\}$ be a central arrangement of hyperplanes in $\C^{\ell}$. Then
\begin{equation}\label{eq:falktheorem1}
\phi_3=2\binom{n+1}{3}-n\dim(A^2)+\dim(A^3_2).
\end{equation}
\end{theorem}
\begin{remark}\label{rem:falkinvariantreduct} Since $\dim(A^3_2)=\dim((E/I_2)^3)=\dim(E^3)-\dim((I_2)^3)$ and $\dim(E^3)=\binom{n}{3}$, then we obtain
\begin{equation}\label{eq:falktheorem}
\phi_3=2\binom{n+1}{3}-n\dim(A^2)+\binom{n}{3}-\dim((I_2)^3).
\end{equation}
\end{remark}

From \cite{falk1988minimal}, we have that $\phi_3$ can also be described from the lower central series of the fundamental group $\pi(M)$ of $M=\C^{\ell}\setminus\bigcup_{H\in\A}H$ the complement of the arrangement. In particular, if we consider the lower central series as a chain of normal subgroups $N_i$, for $k \ge 1$, where $N_1 = \pi(M)$ and $N_{k+1} = [N_k,N_1]$, the subgroup generated by commutators of elements in $N_k$ and $N_1$, then $\phi_3$ is the rank of the finitely generated abelian group $N_3/N_4$. 

\section{Gain graphs}
In this section, we recall the basic notions of gain graphs, and we describe the connection between hyperplane arrangements and gain graphs. See \cite{suyama2019signed, torielli2018freeness, zaslavsky1989biased, zaslavsky1991biased, zaslavsky2003biased} for a thorough treatment of the subject.
\subsection{Gain graphs}

\begin{definition} Let $K$ be a field. A \textbf{gain graph} $\G= (G, \varphi)$ consists of an underlying oriented graph $|\G|=G=(\mathcal{V}_G,\mathcal{E}_G)$ and a \textbf{gain map} $\varphi\colon \mathcal{E}_G\to K^*$ from the edges of $G$ into the \textbf{gain group} $K^*:=K\setminus\{0\}$.  %To be precise we may call $\G$ a \textbf{$K^*$-gain graph}.
\end{definition}
%In the rest of the paper, we will describe $\mathfrak{G}$ with multiplicative notation. 
%Since $\varphi(\mathtt{e}^{-1})=\varphi(\mathtt{e})^{-1}$, then $\varphi(\mathtt{e})$ depends on the orientation of $\mathtt{e}$ but neither orientation is preferred.
By convention $\varphi(\mathtt{e}^{-1})=\varphi(\mathtt{e})^{-1}$, where $\mathtt{e}^{-1}$ means $\mathtt{e}$ with its orientation reversed.
\begin{remark} A signed graph is a gain graph with gain group $\mathfrak{G}$ equal to $\{1, -1\}$. For more details, we refer the reader to \cite{guo2017global},  \cite{guo2017falkinvar}, and \cite{guo2017falk}.
\end{remark}

A \textbf{subgraph} of $\G$ is a subgraph of the underlying graph $|\G|$ with the same gain map, restricted to the subgraph's edges.

%Formally, we may say that $\varphi$ defines a homomorphism $\mathfrak{F}(\mathcal{E}_G)\to\mathfrak{G}$ from the free group on $\mathcal{E}_G$, into the gain group. 
An oriented path $P = \mathtt{e}_1\mathtt{e}_2\cdots \mathtt{e}_k$ has gain value $\varphi(P) = \varphi(\mathtt{e}_1) \varphi(\mathtt{e}_2)\cdots\varphi(\mathtt{e}_k)$ under $\varphi$. %If $P$ is a circle, its gain depends on the starting point and direction, but whether or not the gain equals the identity element $1$ is an absolute. 
An oriented circle whose gain value is $1$ is called \textbf{balanced}. It is \textbf{unbalanced} if it is not balanced. The class of balanced circles is denoted by $\B(\G)$. We call $\left<\G\right>=(G,\B(\G))$ the \textbf{biased graph} associated to $\G$. Clearly $\left<\G\right>$ depends only on the underlying unoriented graph, independent of the chosen orientation of $\G$. We call $\G$ \textbf{balanced} if all its circles are balanced, and \textbf{contrabalanced} if it contains no balanced circles at all. 

In this paper, we will assume that all $2$-circles and loops of $\G$ are unbalanced. 
%\begin{proposition}\cite[Proposition 5.1]{zaslavsky1989biased} If $\G$ is a gain graph, then $[\G]$ is a biased graph.
%\end{proposition}
%So every gain graph is a biased graph. However, the converse is false, see Example 5.8 from \cite{zaslavsky1989biased}.
\begin{figure}[h]
\centering
\begin{tikzpicture}[baseline=10pt]
\draw (1.7,3) node[v,label=right:{$v_1$}](1){};
\draw (0,0) node[v,label=left:{$v_2$}](2){};
\draw (3.4,0) node[v,label=right:{$v_3$}](3){};
\draw[] (1)--(2);
\draw[] (1)--(3);
\draw[] (2)--(3);
\draw[bend right,->>] (1) to (2);
\draw[bend left,->>] (1) to (2);
\draw[bend right,<<-] (2) to (3);
\draw[scale=2,->>] (1)  to[in=135,out=45,loop] (1);
\draw (0.1,1.8) node {2};
\draw (0.5,1.55) node {1};
\draw (1,1.3) node {3};
\draw (2.4,1.3) node {1};
\draw (1.7,0.2) node {1};
\draw (1.7,-0.7) node {2};
\draw (1.2,3.4) node {-1};
\end{tikzpicture}
\caption{Example of a gain graph.}\label{Fig:exgaingraph}
\end{figure}

\begin{example}\label{ex:gaingraphexamp} In Figure \ref{Fig:exgaingraph}, we have a gain graph $\G$ with gains in $\mathbb{Q}^*$, the multiplicative group of rational numbers. The arrows on the edges are there to show the direction in which the gain is as stated. We adopt the simplified notation $\mathtt{e}_{ij}(g)$ for an edge $\{v_i,v_j\}$ with gain  $\varphi(\mathtt{e}_{ij}(g))=g$. (Then for instance $\mathtt{e}_{12}(2)=\mathtt{e}_{21}(2^{-1})$.) The balanced circles are $C_1:=\{\mathtt{e}_{12}(1),\mathtt{e}_{23}(1),\mathtt{e}_{13}(1)\}$ and $C_2:=\{\mathtt{e}_{12}(2),\mathtt{e}_{32}(2),\mathtt{e}_{13}(1)\}$. In fact their gains are $\varphi(C_1)=1\cdot1\cdot1=1$ and $\varphi(C_2)=2\cdot2^{-1}\cdot1=1$. Therefore $\left<\G\right>=(G,\{C_1,C_2\})$.
\end{example}

\begin{theorem}\cite[Theorem 2.1]{zaslavsky1991biased} \label{theo:biasedmatroiddef} Let $\G$ be a gain graph. Then there is a matroid $M(\G)$, whose points are the edges of $\G$ and whose circuits are the edge sets of balanced circles, contrabalanced theta graphs, contrabalanced loose handcuffs and contrabalanced tight handcuffs.
\end{theorem}
\begin{figure}[h]
\centering
\subfigure[]
{
\begin{tikzpicture}[baseline=0]
\foreach \x in {0,...,4}
\draw (90+72*\x:0.8) node[v](\x){};
\draw (0)--(1);
\draw (1)--(2);
\draw (2)--(3)--(4);
\draw (4)--(0);
\draw (1)--(4);
\end{tikzpicture}
}
\hspace{14mm}
\subfigure[]
{
\begin{tikzpicture}[baseline=0]
\foreach \x in {0,...,4}
\draw (180+72*\x:0.8) node[v](\x){};
\draw (-1.8,0) node[v](5){};
\draw (-2.8,0) node[v](6){};
\draw (0)--(1)--(2);
\draw (2)--(3);
\draw (3)--(4);
\draw (4)--(0);
\draw (0)--(5);
\draw (5)--(6);
\draw[scale=2] (6)  to[in=215,out=135,loop] (6);
\end{tikzpicture}
}
\hspace{14mm}
\subfigure[]
{
\begin{tikzpicture}[baseline=0]
\draw (0,0) node[v](1){};
\draw (-0.71, 0.71) node[v](2){};
\draw (-0.71,-0.71) node[v](3){};
\draw ( 0.71, 0.71) node[v](4){};
\draw ( 1.42, 0) node[v](5){};
\draw ( 0.71,-0.71) node[v](6){};
\draw (2)--(1)--(3);
\draw (2)--(3);
\draw (4)--(1)--(6)--(5);
\draw (4)--(5);
\end{tikzpicture}
}
\caption{Examples of (a) theta graph, (b) loose handcuff, (c) tight handcuff.}\label{Fig:thetahandcuff}
\end{figure}

\begin{definition} Let $\G$ be a gain graph. Then the matroid $M(\G)$ is called the \textbf{frame matroid} associated to $\G$.
\end{definition}
\begin{definition} A set of edges of a gain graph $\G$ is called a \textbf{circuit} if the corresponding points form a circuit in $M(\G)$.
\end{definition}

Let $\lambda\colon \mathcal{V}_G\to \mathfrak{G}$ be any function. Switching $\G$ by $\lambda$ means replacing $\varphi(\mathtt{e})$ by $\varphi^\lambda(\mathtt{e}) := \lambda(v)^{-1}\varphi(\mathtt{e})\lambda(w)$, where $\mathtt{e}$ is oriented from $v$ to $w$. The switched graph, $\G^\lambda = (G, \varphi^\lambda),$ is called \textbf{switching equivalent} to $\G$. In general, we will denote by $[\G]$ any gain graph that is switching equivalent to $\G$ for some $\lambda$.
\begin{lemma}\cite[Lemma 5.2]{zaslavsky1989biased}\label{lemma:samebiasmatr} $\langle[\G]\rangle=\langle\G\rangle$.
\end{lemma}
\begin{lemma}\cite[Lemma 5.3]{zaslavsky1989biased} $\G=(G,\varphi)$ is balanced if and only if $\varphi$ switches to the identity gain.
\end{lemma}
%By the previous two lemmas, similarly to the case of signed graph (see Proposition 3.9 in \cite{guo2017falk}), we have the following
%\begin{proposition} Two gain graphs with the same underlying graph are switching equivalent if and only if they have the same list of balanced circles, contrabalanced theta graphs and contrabalanced (loose and tight) handcuffs.
%\end{proposition}

Directly from Theorem \ref{theo:biasedmatroiddef}, we have the following result.
\begin{proposition}\label{prop:biaseqsamematroid}
If $\G_1$ and $\G_2$ are two gain graphs such that $\langle\G_1\rangle=\langle\G_2\rangle$, then $M(\G_1)=M(\G_2)$.
\end{proposition}

By Proposition \ref{prop:biaseqsamematroid} and Lemma \ref{lemma:samebiasmatr}, we have the following
\begin{corollary}\label{corol:switchinggivesamematroid} $M([\G])=M(\G)$. 
\end{corollary}

\subsection{Hyperplane arrangement realizations of gain graphs}
In this subsection, we will consider $K$ a field, $\G= (G, \varphi)$ a gain graph with gain group $K^*$, and $\mathcal{V}_G=\{1,\dots,\ell\}$. 
%Moreover, we will assume that all $2$-circles and loops of $\G$ are unbalanced. 
\begin{definition} Let $\A(\G)$ be the hyperplane arrangement in $K^\ell$ consisting of the following hyperplanes
$$\{x_i=\varphi(\mathtt{e}_{ij})x_j\}\ \text{ for } \mathtt{e}_{ij}\in \mathcal{E}_G.$$
We will call $\A(\G)$ the \textbf{canonical linear hyperplane representation} of $\G$.
\end{definition}
Notice that since we assume that every loop is unbalanced, then if $\mathtt{e}_{ii}$ is a loop, we have $\varphi(\mathtt{e}_{ii})\ne1$, and hence we attach to it the hyperplane $\{x_i=0\}$. Moreover, since all $2$-circles are unbalanced, the hyperplanes of $\A(\G)$ are all distinct.
\begin{example} Consider the gain graph described in Example \ref{ex:gaingraphexamp}. Then we obtain the hyperplane arrangement $\A(\G)\subseteq\mathbb{R}^3$ with defining equation $x(x-y)(x-2y)(x-3y)(y-z)(2y-z)(x-z)$.
\end{example}
Given a gain graph $\G$, we can now associate to it two matroids: the frame matroid and the matroid associated to the intersection lattice of $\A(\G)$. In \cite{zaslavsky2003biased}, Zaslavsky proved that these two matroids coincide. In particular, he proved the following.
\begin{theorem}\cite[Corollary 2.2]{zaslavsky2003biased}\label{theo:GandAGmatroid} $M(\G)\cong M(\A(\G))$.
\end{theorem}

%More in general we have that
\begin{proposition}\label{prop:biaseqsamefalkinv}
Let $\G_1$ and $\G_2$ be two gain graphs such that $\langle\G_1\rangle=\langle\G_2\rangle$. Then $\phi_3(\A(\G_1)) = \phi_3(\A(\G_2))$.
\end{proposition}
\begin{proof} By Proposition~\ref{prop:biaseqsamematroid} and Theorem~\ref{theo:GandAGmatroid}, $M(\A(\G_1))\cong M(\A(\G_2))$. This implies that $\A(\G_1)$ and $\A(\G_2)$ have isomorphic Orlik--Solomon algebra, and hence they have the same Falk invariant $\phi_3$.
\end{proof}

Similarly as in the case of signed graph (see Corollary 3.11 in \cite{guo2017falk}), by Proposition~\ref{prop:biaseqsamefalkinv} and Lemma~\ref{lemma:samebiasmatr} we have the following
\begin{corollary}\label{corol:switchingsamefalkinv} Let $\G_1$ and $\G_2$ be two switching equivalent gain graphs. Then $\phi_3(\A(\G_1)) = \phi_3(\A(\G_2))$.
\end{corollary}
%\begin{proof}
%Since $\G_1$ and $\G_2$ are switching equivalent, then $\G_1=(G,\varphi_1)$ and $\G_2=(G,\varphi_2)$. Let $\lambda\colon \mathcal{V}_G\to \mathfrak{G}$ be a switching function such that $\G_2=(\G_1)^\lambda$. For each $\mathtt{e}_{ij}\in\mathcal{ E}_G$, %which is the edge oriented from $i$ to $j,$
% the hyperplane $x_i=\varphi_1(\mathtt{e_{ij}})x_j$ in $\A(\G_1)$ corresponds to the hyperplane $x_i=\varphi_2(\mathtt{e}_{ij})x_j$ in 
%$\A(\G_2)$, where $\varphi_2(\mathtt{e}_{ij})=\varphi_1^\lambda(\mathtt{e}_{ij})=\lambda(v_i)^{-1}\varphi_1(\mathtt{e}_{ij})\lambda(v_j)$.
%Hence the switching function $\lambda$ describes a change of coordinates to pass from $\A(\G_1)$ to $\A(\G_2)$. This implies that the two arrangements have isomorphic associated matroid. So also their Orlik--Solomon ideals and algebras are isomorphic, and their Falk invariants are equal. 
%\end{proof}

\section{List of distinguished biased graphs}\label{sect:listgaingraphimportant}
In this section, we will describe all the gain graphs that we need to express our main theorem. Since we will consider $K^*=\mathbb{Q}^*$, we will describe the underlying graph, together with the list of balanced circles. 
%Hence, we will describe the underlying graph, labeling only the edges, with the list of balanced circles.  %We adopt the simplified notation $ijk$ for a circle $e_ie_je_k$. %Moreover, we call a circle \textbf{distinguish} if it is a balanced circle or a contrabalanced theta graph, or a contrabalanced (loose or tight) handcuff. 
%We will take their names from the nomenclature used in \cite{guo2017falk}.

\begin{figure}[h!]
\centering
\subfigure[]
{
\begin{tikzpicture}[baseline=10pt]
\draw (0,1.5) node[v](1){};
\draw (3,1.5) node[v](2){};
\draw (0,1) node {$v_1$};
\draw (3,1) node {$v_2$};
\draw[scale=2,->>] (1)  to[in=225,out=135,loop] (1);
\draw[bend right, ->>] (1) to (2);
\draw[bend left, <<-] (1) to (2);
\draw (1.5,2.2) node {2};
\draw (1.5, 0.8) node {3};
\draw (-0.4,2) node {-1};
\end{tikzpicture}
}
\hspace{14mm}
%\subfigure[]
%{
%\begin{tikzpicture}[baseline=10pt]
%\draw (0,1.5) node[v](1){};
%\draw (3,1.5) node[v](2){};
%\draw[scale=2] (1)  to[in=225,out=135,loop] (1);
%\draw[scale=2] (2)  to[in=45,out=315,loop] (2);
%\draw[] (1)--(2);
%\draw (-0.4,2) node {1};
%\draw (1.5,1.8) node {2};
%\draw (3.4,2) node {3};
%\end{tikzpicture}
%}
%\hspace{14mm}
\subfigure[]
{
\begin{tikzpicture}[baseline=10pt]
\draw (0,1.5) node[v](1){};
\draw (3,1.5) node[v](2){};
\draw (0,1) node {$v_1$};
\draw (3,1) node {$v_2$};
\draw[scale=2,->>] (1)  to[in=225,out=135,loop] (1);
\draw[scale=2,->>] (2)  to[in=45,out=315,loop] (2);
\draw[bend right,->>] (1) to (2);
\draw[bend left,<<-] (1) to (2);
\draw (-0.4,2) node {-1};
\draw (1.5,2.2) node {2};
\draw (1.5, 0.8) node {3};
\draw (3.4,2) node {-1};
\end{tikzpicture}
}
\hspace{14mm}
%\subfigure[]
%{
%\begin{tikzpicture}[baseline=10pt]
%\draw (0,3) node[v](1){};
%\draw (0,0) node[v](2){};
%\draw (3,0) node[v](3){};
%\draw (3,3) node[v](4){};
%\draw[] (1)--(2);
%\draw[] (1)--(3);
%\draw[] (1)--(4);
%\draw[] (2)--(3);
%\draw[] (2)--(4);
%\draw[] (3)--(4);
%\draw (-0.2,1.5) node {1};
%\draw (1.5, 3.2) node {2};
%\draw (1.8, 2.1) node {3};
%\draw (3.2,1.5) node {6};
%\draw (1.5, -0.2) node {4};
%\draw (1.8, 0.9) node {5};
%\end{tikzpicture}
%}
%\hspace{14mm}
\subfigure[]
{
\begin{tikzpicture}[baseline=10pt]
\draw (1.7,3) node[v,label=right:{$v_1$}](1){};
\draw (0,0) node[v,label=left:{$v_2$}](2){};
\draw (3.4,0) node[v,label=right:{$v_3$}](3){};
\draw[] (1)--(2);
\draw[] (1)--(3);
\draw[] (2)--(3);
\draw[bend right,<<-] (1) to (2);
\draw[bend left,->>] (1) to (3);
\draw[bend right,<<-] (2) to (3);
\draw (0.1,1.7) node {-1};
\draw (1,1.3) node {1};
\draw (2.4,1.3) node {1};
\draw (3.3,1.7) node {-1};
\draw (1.7,0.2) node {1};
\draw (1.7,-0.7) node {-1};
\end{tikzpicture}
}
%\hspace{14mm}
%\subfigure[]
%{
%\begin{tikzpicture}[baseline=10pt]
%\draw (1.7,3) node[v](1){};
%\draw (0,0) node[v](2){};
%\draw (3.4,0) node[v](3){};
%\draw[] (1)--(2);
%\draw[] (1)--(3);
%\draw[] (2)--(3);
%\draw[scale=2] (1)  to[in=135,out=45,loop] (1);
%\draw[scale=2] (2)  to[in=225,out=135,loop] (2);
%\draw[scale=2] (3)  to[in=45,out=315,loop] (3);
%\draw (1,1.3) node {1};
%\draw (2.4,1.3) node {2};
%\draw (1.7,0.3) node {3};
%\draw (1.2,3.4) node {4};
%\draw (3.8,0.5) node {5};
%\draw (-0.4,0.5) node {6};
%\end{tikzpicture}
%}
\caption{List of underlying graphs.}\label{Fig:disgraphs1}
\end{figure}
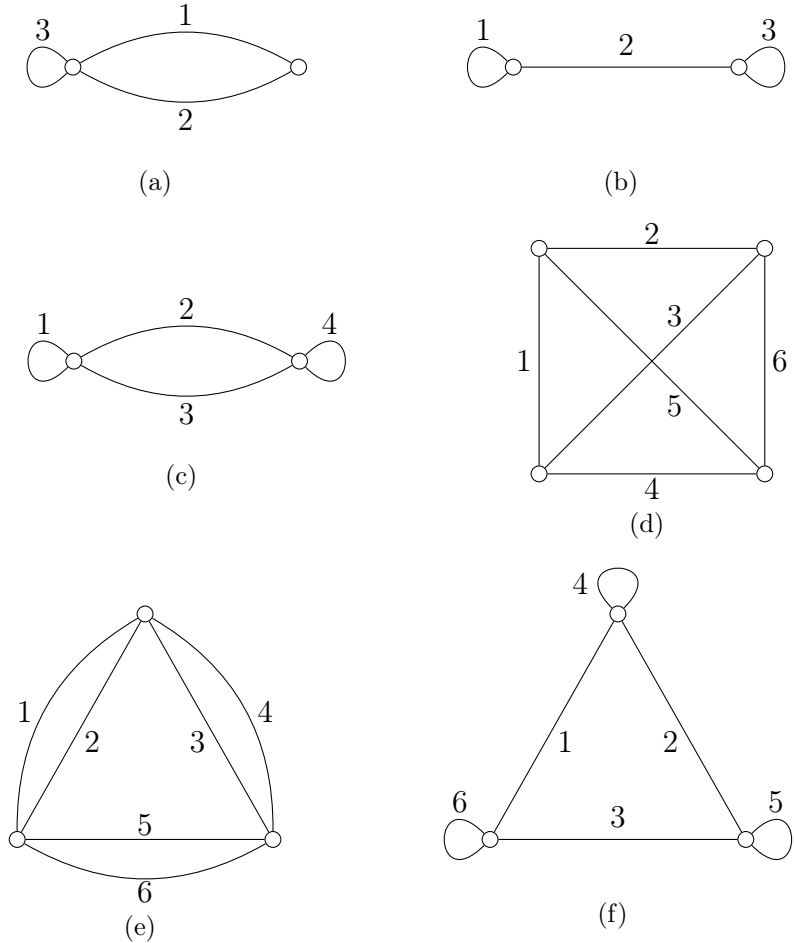

\begin{figure}[h!]
\centering
\subfigure[]
{
\begin{tikzpicture}[baseline=10pt]
\draw (1.7,3) node[v,label=right:{$v_1$}](1){};
\draw (0,0) node[v,label=left:{$v_2$}](2){};
\draw (3.4,0) node[v,label=right:{$v_3$}](3){};
\draw[] (1)--(2);
\draw[] (1)--(3);
\draw[] (2)--(3);
\draw[bend right,<<-] (1) to (2);
\draw[bend left,->>] (1) to (3);
\draw[scale=2,->>] (1)  to[in=135,out=45,loop] (1);
\draw (0.1,1.7) node {-1};
\draw (1,1.3) node {1};
\draw (2.4,1.3) node {1};
\draw (3.3,1.7) node {-1};
\draw (1.7,0.2) node {1};
\draw (1.2,3.4) node {-1};
\end{tikzpicture}
}
\hspace{14mm}
\subfigure[]
{
\begin{tikzpicture}[baseline=10pt]
\draw (1.7,3) node[v,label=right:{$v_1$}](1){};
\draw (0,0) node[v,label=left:{$v_2$}](2){};
\draw (3.4,0) node[v,label=right:{$v_3$}](3){};
\draw[] (1)--(2);
\draw[] (1)--(3);
\draw[] (2)--(3);
\draw[bend right,<<-] (1) to (2);
\draw[bend left,->>] (1) to (3);
\draw[bend right,<<-] (2) to (3);
\draw[scale=2,->>] (1)  to[in=135,out=45,loop] (1);
\draw (0.1,1.7) node {-1};
\draw (1,1.3) node {1};
\draw (2.4,1.3) node {1};
\draw (3.3,1.7) node {-1};
\draw (1.7,0.2) node {1};
\draw (1.7,-0.7) node {-1};
\draw (1.2,3.4) node {-1};
\end{tikzpicture}
}
\hspace{14mm}
\subfigure[]
{
\begin{tikzpicture}[baseline=10pt]
\draw (1.7,3) node[v,label=right:{$v_1$}](1){};
\draw (0,0) node[v,label=left:{$v_2$}](2){};
\draw (3.4,0) node[v,label=right:{$v_3$}](3){};
\draw[<<-] (1)--(2);
\draw[] (1)--(3);
\draw[] (2)--(3);
%\draw[bend right,->>] (1) to (2);
\draw[->>] (1) .. controls (-1,2) .. (2);
\draw[bend left,<<-] (1) to (3);
\draw[bend right,<<-] (2) to (3);
\draw[bend right] (1) to (2);
\draw[bend right,->>] (1) to (3);
\draw (-0.3,1.8) node {2};
\draw (0.5,1.55) node {1};
\draw (1,1.3) node {2};
\draw (2.4,1.3) node {2};
\draw (2.8,1.55) node {1};
\draw (3.3,1.8) node {2};
\draw (1.7,0.2) node {1};
\draw (1.7,-0.7) node {2};
\end{tikzpicture}
}
\hspace{14mm}
\subfigure[]
{
%\begin{tikzpicture}[baseline=10pt]
%\draw (1.7,3) node[v,label=right:{$v_1$}](1){};
%\draw (0,0) node[v,label=left:{$v_2$}](2){};
%\draw (3.4,0) node[v,label=right:{$v_3$}](3){};
%\draw[] (1)--(2);
%\draw[] (1)--(3);
%\draw[] (2)--(3);
%\draw[bend right,<<-] (1) to (2);
%\draw[bend left,<<-] (1) to (3);
%\draw[bend right,->>] (2) to (3);
%\draw[bend left,->>] (1) to (2);
%\draw[bend right,->>] (1) to (3);
%\draw[bend left,<<-] (2) to (3);
%\draw (0.1,1.8) node {2};
%\draw (0.5,1.55) node {1};
%\draw (1,1.3) node {2};
%\draw (2.4,1.3) node {2};
%\draw (2.8,1.55) node {1};
%\draw (3.3,1.8) node {2};
%\draw (1.7,0.3) node {2};
%\draw (1.7,-0.2) node {1};
%\draw (1.7,-0.7) node {4};
%\end{tikzpicture}
\begin{tikzpicture}[baseline=10pt]
\draw (1.7,3) node[v,label=right:{$v_1$}](1){};
\draw (0,0) node[v,label=left:{$v_2$}](2){};
\draw (3.4,0) node[v,label=right:{$v_3$}](3){};
\draw[<<-] (1)--(2);
\draw[] (1)--(3);
\draw[<<-] (2)--(3);
%\draw[bend right,->>] (1) to (2);
\draw[->>] (1) .. controls (-1,2) .. (2);
\draw[bend left,<<-] (1) to (3);
\draw[bend right] (2) to (3);
\draw[bend right] (1) to (2);
\draw[bend right,->>] (1) to (3);
\draw (-0.3,1.8) node {2};
\draw (0.5,1.55) node {1};
\draw (1,1.3) node {2};
\draw (2.4,1.3) node {2};
\draw (2.8,1.55) node {1};
\draw (3.3,1.8) node {2};
\draw (1.7,0.2) node {2};
\draw (1.7,-0.3) node {1};
\draw (1.7,-0.9) node {4};
\draw[->>] (2) .. controls (1.7,-1.5) .. (3);
\end{tikzpicture}

}
\caption{List of underlying graphs.}\label{Fig:disgraphs2}
\end{figure}
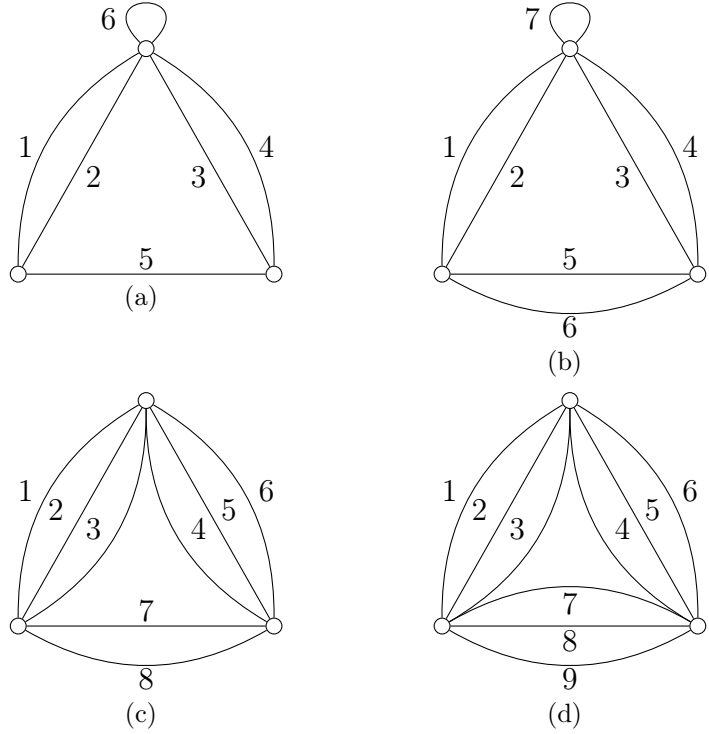
\begin{enumerate}
\item The biased graph $1K_n$ has as underlying graph the complete graph on $n$ vertices and it is a balanced.
\item The biased graph $1K_n^\circ$ is the graph $1K_n$ with an unbalanced loop at every vertex.%has as underlying graph the one depicted in Figure \ref{Fig:disgraphs1}(b) and as circuit $\B:=\{123\}$.
\item The biased graph $\pm 1K_2^{(1)}$ is the one associated with the gain graph depicted in Figure \ref{Fig:disgraphs1}(a) and it is contrabalanced. 
\item The biased graph $\pm 1K_2^\circ$ is the one associated with the gain graph depicted in Figure \ref{Fig:disgraphs1}(b) and it is contrabalanced. 
%\item The gain graph $K_4$ has as underlying graph the one depicted in Figure \ref{Fig:disgraphs1}(d) and as distinguish circles $\B:=\{123,145,256,346\}$. 
\item The biased graph $\pm 1K_3$ is the one associated with the gain graph depicted in Figure \ref{Fig:disgraphs1}(c). %The list of its balanced circles is $\{\mathtt{e}_{12}(1)\mathtt{e}_{23}(1)\mathtt{e}_{31}(1),$ $\mathtt{e}_{21}(-1)\mathtt{e}_{13}(-1)\mathtt{e}_{23}(-1),$ $ \mathtt{e}_{21}(-1)\mathtt{e}_{31}(1)\mathtt{e}_{32}(-1),$ $\mathtt{e}_{12}(1)\mathtt{e}_{13}(-1)\mathtt{e}_{32}(-1)\}.$
%\item The gain graph $[1K_3^\circ]$ has as underlying graph the one depicted in Figure \ref{Fig:disgraphs1}(f) and as circuits $\B:=\{123,146,245,365\}$. 
\item The biased graph $\pm 1K_3^{(1)}\setminus \mathtt{e}_{23}(-1)$is the one associated with the gain graph depicted in Figure \ref{Fig:disgraphs2}(a).% Its balanced circles are $\{\mathtt{e}_{21}(-1)\mathtt{e}_{13}(-1)\mathtt{e}_{23}(1),$ $\mathtt{e}_{12}(1)\mathtt{e}_{23}(1)\mathtt{e}_{31}(1)\}$. 
\item The biased graph $\pm 1K_3^{(1)}$ is the one associated with the gain graph depicted in Figure \ref{Fig:disgraphs2}(b). %Its balanced circles are $\{\mathtt{e}_{21}(-1)\mathtt{e}_{13}(-1)\mathtt{e}_{23}(1),$ $\mathtt{e}_{12}(1)\mathtt{e}_{23}(1)\mathtt{e}_{31}(1),$ $\mathtt{e}_{21}(-1)\mathtt{e}_{31}(1)\mathtt{e}_{32}(-1),$ $\mathtt{e}_{12}(1)\mathtt{e}_{13}(-1)\mathtt{e}_{32}(-1)\}$. 
\item The biased graph $\mathbb{Z}_3K_3\setminus \mathtt{e}$ is the one associated with the gain graph depicted in Figure \ref{Fig:disgraphs2}(c). %Its balanced circles are $\{\mathtt{e}_{12}(1)\mathtt{e}_{23}(1)\mathtt{e}_{31}(1),$ $\mathtt{e}_{12}(2)\mathtt{e}_{23}(1)\mathtt{e}_{13}(2),$ $\mathtt{e}_{12}(2)\mathtt{e}_{32}(2)\mathtt{e}_{31}(1)$, $\mathtt{e}_{12}(1)\mathtt{e}_{32}(2)\mathtt{e}_{31}(2),$ $\mathtt{e}_{21}(2)\mathtt{e}_{31}(2)\mathtt{e}_{23}(1)\}$.
\item The biased graph $\mathbb{Z}_3K_3$ is the one associated with the gain graph depicted in Figure \ref{Fig:disgraphs2}(d). %Its balanced circles are $\{\mathtt{e}_{12}(1)\mathtt{e}_{23}(1)\mathtt{e}_{31}(1),$ $\mathtt{e}_{21}(2)\mathtt{e}_{31}(2)\mathtt{e}_{23}(1),$ $\mathtt{e}_{21}(2)\mathtt{e}_{31}(1)\mathtt{e}_{23}(2),$ $\mathtt{e}_{12}(1)\mathtt{e}_{13}(2)\mathtt{e}_{23}(2),$ $\mathtt{e}_{12}(2)\mathtt{e}_{13}(2)\mathtt{e}_{23}(1),$ $\mathtt{e}_{21}(2)\mathtt{e}_{13}(2)\mathtt{e}_{23}(4)\}$. 
\end{enumerate}
%Furthermore, if $\G$ is a gain graph we denote by $\overline{\G}$ a gain graph switching equivalent to $\G$ for some switching function $\lambda$.

Notice that by construction we have the following
\begin{lemma}\label{lemma:g1g2notd3inside} The graphs $\mathbb{Z}_3K_3\setminus \mathtt{e}$ and $\mathbb{Z}_3K_3$ do not have any subgraphs isomorphic to $\pm 1K_3$.
\end{lemma}
\begin{remark}
The graphs $1K_3,$ $1K_2^\circ,$ and $\pm 1K_2^{(1)}$ have isomorphic frame matroids, and the frame matroids of the graphs $1K_4$, $1K_3^\circ$, $\pm 1K_3 $, and $\pm 1K_3^{(1)}\setminus \mathtt{e}_{23}(-1)$ are also isomorphic, see Section $6.$
\end{remark}

\section{Main theorem}
In this section, we will describe how to compute the Falk invariant $\phi_3$ for $\mathcal{A}(\G)$, an arrangement associated to a gain graph $\G$ such that $\langle\G\rangle$ does not have a subgraph isomorphic to $\pm 1K_2^\circ$, it has at most triple parallel edges and it has no loops adjacent to a theta graph with three edges. Moreover, we will assume that all $2$-circles and loops of $\G$ are unbalanced.  

In the remainder of the paper, to fix the notation, we will suppose $\G$ is a gain graph whose underlying graph $|\G|$ is on $\ell$ vertices having $n$ edges. Since our result will only depend on $\langle\G\rangle$ and not on the specific gain value of the edges (see Proposition~\ref{prop:biaseqsamefalkinv}), we will label the edges of $|\G|$ as elements of $[n]:=\{1,\dots, n\}$. 
%Furthermore, when we will discuss isomorphic subgraphs, we intend isomorphic as biased graphs (see Definition \ref{def:isobiasedgraphs}).

We define the numbers of some subgraphs of a graph $\langle\G\rangle$ as the following:
\begin{itemize}
\item[] $k_l$ denotes the number of subgraphs of $\langle\G\rangle$ isomorphic to a $1K_l$; 
\item[] $k_3^{\pm}$ denotes the number of subgraphs of $\langle\G\rangle$ isomorphic to $\pm 1K_3$ but not contained in $\pm 1K_3^{(1)}$;
%\item $k_2^{(1)}$ denotes the number of subgraphs of $\langle\G\rangle$ isomorphic to $\langle\pm 1K_2^{(1)}\rangle$,
\item[] $k_l^{(1)}$ denotes the number of subgraphs of $\langle\G\rangle$ isomorphic to $\pm 1K_l^{(1)}$;
\item[] $k_l^\circ$ denotes the number of subgraphs of $\langle\G\rangle$ isomorphic to a $1K_l^\circ$;
\item[] $g_0$ denotes the number of subgraphs of $\langle\G\rangle$ isomorphic to a $\pm 1K_3^{(1)}\setminus \mathtt{e}_{23}(-1)$ but not contained in $\pm 1K_3^{(1)}$;
\item[] $g_1$ denotes the number of subgraphs of $\langle\G\rangle$ isomorphic to a $\mathbb{Z}_3K_3\setminus \mathtt{e}$ but not contained in $\mathbb{Z}_3K_3$;
\item[] $g_2$ denotes the number of subgraphs of $\langle\G\rangle$ isomorphic to a $\mathbb{Z}_3K_3$;
\item[] $\Theta$ counts contrabalanced triple parallel edges.%denotes the number of subgraphs isomorphic to a contrabalanced theta graph with only three edges.
\end{itemize}

The goal of this section is to prove the following theorem.
\begin{theorem}\label{theo:ourmain}
 %For an arrangement associated to a gain graph $\langle\G\rangle$ that does not have a subgraph isomorphic to $\pm 1K_2^\circ$, it has no loops adjacent to a theta graph with only three edges and it has at most triple parallel edges, we have
For an arrangement associated to a gain graph $\G$ such that $\langle\G\rangle$ has no subgraph isomorphic to $\pm 1K_2^\circ$ or to a contrabalanced triple parallel edge with an adjacent loop and has at most triple parallel edges, we have
\begin{equation}\label{eq:ourmainformula}
\phi_3=2(k_3+k_4+k_3^{\pm}+k_2^{(1)}+k_2^\circ+k_3^\circ+g_0+g_2+\Theta)+5k_3^{(1)}+g_1.
\end{equation}

\end{theorem}

\begin{remark} The restrictions on $\G$ in Theorem~\ref{theo:ourmain} are exactly that the associated matroid has no submatroid isomorphic to the $4$-point line. 
For this reason, we will say that $\G$ has no $U_{2,4}$ subgraphs.
\end{remark}

In order to compute $\phi_3$, we will use Theorem \ref{theo:falkinvar}, hence we need firstly to identify the triples $S$ in $[n]$ that are dependent. Clearly, we have the following
\begin{lemma} $S=(i_1, i_2, i_3)$ is dependent if and only if $i_1, i_2, i_3$ correspond to the edges of a subgraph of $\langle\G\rangle$ that is isomorphic to $1K_3$, or $\pm 1K_2^{(1)}$ or a contrabalanced theta graph with only three edges.
 %The subset $S\subseteq \mathcal{E}_G$ containing $3$ edges is dependent if and only if the three edges correspond to the edges of a subgraph of $\G$ that is isomorphic to a balanced triangle $\langle1K_3\rangle$, or a contrabalanced $(\pm 1K_2)^\lambda$ or a contrabalanced triple parallel edge.
\end{lemma}

Since a dependent triple $S$ corresponds to a circuit of size $3$ in $M(\A(\G))$, we call such S a \textbf{$3$-circuit}. Moreover, we will write
$$\mathcal{C}_3:=\spann\{e_S\in E~|~S \text{ is a $3$-circuit}\}$$
which is a subset of $E$ as a vector space over $\C$.

\begin{remark} Notice that the contrabalanced theta graphs, loose handcuffs, and tight handcuffs are subdivisions of the unbalanced $3$- circuits. In particular, If $\G_1$ and $\G_2$ are two switching equivalent gain graphs with the same underlying graph, then $\mathcal{C}_3(\G_1)=\mathcal{C}_3(\G_2)$, see Proposition \ref{prop:biaseqsamematroid}.
\end{remark}

Since $e_ie_je_k=-e_je_ie_k$, it is clear that the dimension of the vector space $\mathcal{C}_3$ is $k_3+k_2^{(1)}+k_2^\circ+\Theta$. Let $C'_3$ be a basis of $\mathcal{C}_3$ consisting of monomials in one-to-one correspondence with the subgraphs of $\langle\G\rangle$ isomorphic to a $1K_3$, or a $\pm 1K_2^{(1)}$ or a $1K_2^\circ$ or a contrabalanced theta graph with only three edges.

\begin{lemma}\label{lem:dimA2} $\dim(A^2)=\binom{n}{2}-k_3-k_2^{(1)}-k_2^\circ-\Theta$.
\end{lemma}
\begin{proof} By definition $A=E/I$, hence 
$$\dim(A^2)=\dim(E^2)-\dim(I^2)=\binom{n}{2}-\dim(I^2).$$
By construction $I^2=\spann\{\partial e_{ijk}~|~e_{ijk}\in\mathcal{C}_3\}$. Notice that if $e_1,\dots, e_s$ are distinct elements of $C'_3$ then the corresponding subgraphs of $\langle\G\rangle$ share at most one edge and hence $\partial e_1, \dots, \partial e_s$ are linearly independent. This implies that  $\dim(I^2)=\dim(\mathcal{C}_3)=k_3+k_2^{(1)}+k_2^\circ+\Theta$, and the thesis follows.
% each element $\partial e_{ijk}\in\mathcal{C}_3$ will not appear in other elements from another $3$-circuits,
\end{proof}
Using Theorem \ref{theo:falkinvar} and Remark \ref{rem:falkinvariantreduct}, to prove Theorem \ref{theo:ourmain}, we just need to describe $\dim((I_2)^3)$. To do so, consider 
$$C_3:=\{e_t\partial e_{ijk}~|~e_{ijk}\in C'_3,t\in\{i,j,k\}\},$$
and
$$F_3:=\{e_t\partial e_{ijk}~|~e_{ijk}\in C'_3,t\in[n]\setminus\{i,j,k\}\}.$$
%where $[n]=\{1, \dots, n\}.$

By construction $(I_2)^3=I^2\cdot E^1=\spann\{e_t\partial e_{ijk}~|~e_{ijk}\in C'_3,t\in[n]\}$, and hence
$$(I_2)^3=\spann(C_3)+\spann(F_3).$$
\begin{lemma}\label{lemma:firstbreakdirectsum} For an arrangement associated to a gain graph $\G$ without $U_{2,4}$ subgraphs,
%loops adjacent to a theta graph with only three edges, with at most triple parallel edges, and such that $\langle\G\rangle$ does not contain a subgraph isomorphic to $\pm 1K_2^\circ$, 
we have
$$(I_2)^3=\spann(C_3)\oplus\spann(F_3).$$
\end{lemma}
\begin{proof} Since $\langle\G\rangle$ does not contain a $\pm 1K_2^\circ$ as subgraph or loops adjacent to a theta graph or quadruple parallel edges, any two $3$-circuits share at most one element. Moreover, if $e_t\partial e_{ijk}\in C_3$, then $e_t\partial e_{ijk}=\pm e_{ijk}$ with $i,j,k$ edges of the same $3$-circuit. On the other hand,  if $e_t\partial e_{ijk}\in F_3$, then $e_t\partial e_{ijk}= e_{tjk}-e_{tik}+e_{tij}$ and $t$ does not belong to the same $3$-circuit as $i,j,k$. This then gives us that $\spann(C_3)\cap\spann(F_3)=\{0\}$.
\end{proof}
\begin{remark} Notice that if we allow $\langle\G\rangle$ to have subgraphs isomorphic to $\pm 1K_2^\circ$ or a loop adjacent to a theta graph or quadruple parallel edges, then the previous lemma is not true any more.
\end{remark}
By Lemma~\ref{lemma:firstbreakdirectsum}, we can write
\begin{align*}
\dim((I_2)^3) & =\dim(\spann(C_3))+\dim(\spann(F_3))\\
& =k_3+k_2^{(1)}+k_2^\circ+\Theta+\dim(\spann(F_3)).
\end{align*}
To prove our main result we need to compute $\dim(\spann(F_3))$. To do so, consider the following sets:
\begin{align*}
F^1_3 := & \{e_t\partial e_{ijk}\in F_3~|~t,i,j,k \text{ are not in the same }1K_4,\pm 1K_3; \\
 & \pm 1K_3^{(1)}\setminus \mathtt{e}_{23}(-1), \pm 1K_3^{(1)}, 1K_3^\circ,\mathbb{Z}_3K_3\setminus \mathtt{e},\mathbb{Z}_3K_3\};\\
F^2_3 := & \{e_t\partial e_{ijk}\in F_3~|~t,i,j,k \text{ are in the same }1K_4\}; \\
F^3_3 := & \{e_t\partial e_{ijk}\in F_3~|~t,i,j,k \text{ are in the same }\pm 1K_3 \text{ not contained }\\
 &\text{ in a }\pm 1K_3^{(1)}\};\\
F^4_3 := & \{e_t\partial e_{ijk}\in F_3~|~t,i,j,k \text{ are in the same }\pm 1K_3^{(1)}\setminus \mathtt{e}_{23}(-1) \text{ not} \\
 &\text{ contained in a }\pm 1K_3^{(1)}\};\\
F^5_3 := & \{e_t\partial e_{ijk}\in F_3~|~t,i,j,k \text{ are in the same }\pm 1K_3^{(1)}\};\\
F^6_3 := & \{e_t\partial e_{ijk}\in F_3~|~t,i,j,k \text{ are in the same } 1K_3^\circ\};\\
F^7_3 := & \{e_t\partial e_{ijk}\in F_3~|~t,i,j,k \text{ are in the same }\mathbb{Z}_3K_3\setminus \mathtt{e} \text{ not contained } \\
  &\text{ in a }\mathbb{Z}_3K_3\};\\
F^8_3 := & \{e_t\partial e_{ijk}\in F_3~|~t,i,j,k \text{ are in the same }\mathbb{Z}_3K_3\}.
\end{align*}

\begin{lemma} For an arrangement associated to a gain graph $\G$ without $U_{2,4}$ subgraphs,
%loops adjacent to a theta graph with only three edges, with at most triple parallel edges, and such that $\langle\G\rangle$ does not contain a subgraph isomorphic to $\pm 1K_2^\circ$, 
we have
$$\spann(F_3)=\bigoplus_{i=1}^8\spann(F^i_3).$$
\end{lemma}
\begin{proof} By hypothesis, $\langle\G\rangle$ does not contain a subgraph isomorphic to $\pm 1K_2^\circ$ or loops adjacent to a theta graph with only three edges or quadruple parallel edges. Moreover, if $e_t\partial e_{ijk}\in F^p_3$ and $e_a\partial e_{bcd}\in F^q_3$ with $p,q\ge2$ and $p\ne q$, then by with Lemma \ref{lemma:g1g2notd3inside} the corresponding graphs share at most $3$ edges. This implies that at least one term of $e_t\partial e_{ijk}\in F^p_3$ appears only in the expression of $e_t\partial e_{ijk}\in F^p_3$ and not in the expression of any other element in $\bigcup_{\substack{2\le i \le 8\\ i\ne p}}(F^i_3)$. So $e_t\partial e_{ijk}$ can not be expressed linearly by the elements of $\bigcup_{\substack{2\le i \le 8\\ i\ne p}}(F^i_3)$.
% This fact together with Lemma \ref{lemma:g1g2notd3inside} implies that $\spann(F^p_3)\cap\spann(F^q_3)=\{0\}$ for all $p,q=2, \dots, 8$ such that $p\ne q$.

For any element $e_t\partial e_{ijk}$ of $F^1_3$, we assert that at least one of the terms $e_{tjk}, e_{tik}, e_{tij}$ appears only in the expression of $e_t\partial e_{ijk}\in F^1_3$ and not in the expression of any other element in $\bigcup_{i=2}^8\spann(F^i_3)$. So $e_t\partial e_{ijk}$ can not be expressed linearly by the elements of $F^2_3, \dots, F^8_3$.

Since the edges $t, i, j, k$ are not in the same $1K_4,\pm 1K_3,\pm 1K_3^{(1)}\setminus \mathtt{e}_{23}(-1),\pm 1K_3^{(1)}, 1K_3^\circ,\mathbb{Z}_3K_3\setminus \mathtt{e},\mathbb{Z}_3K_3$, and we do not consider the graphs having subgraphs isomorphic to $\pm 1K_2^\circ$ or loops adjacent to a theta graph with only three edges or quadruple parallel edges, we should only consider three cases about the edge $t$: it can be adjacent to none of the edges $i,j,k$, to two of them, or to all of them.

Assume that the edge $t$ is adjacent to none of the edges $i,j,k$. This implies that $t$ and none of $i,j,k$ can appear in the same $3$-circuit. Hence any element $e_t\partial e_{ijk}$ of $F^1_3$ will not appear in any of $F^2_3, \dots, F^8_3$.

Assume now that the edge $t$ is adjacent to two of the edges $i,j,k$, then we should consider several possibilities.
Suppose that in the set $\{t,i,j,k\}$ there is no loop. If all the terms of the element $e_t\partial e_{ijk}\in F^1_3$ appear in $F^2_3, \dots, F^8_3$, then $t,i,j,k$ have to appear in the same $1K_4$, but this is impossible by construction.
Suppose that $t$ is a loop and there is no loop in the set $\{i,j,k\}$. If all the terms of the element $e_t\partial e_{ijk}\in F^1_3$ appear in $F^2_3, \dots, F^8_3$, then $t,i,j,k$ have to appear in the same $\pm 1K_3^{(1)}\setminus e_{23}(-1)$ or in the same $\pm 1K_3^{(1)}$, but this is impossible by construction.
Suppose that $t$ is not a loop and there is one loop in the set $\{i,j,k\}$. In this case $i,j,k$ are the edges of a $\pm 1K_2^{(1)}$. Hence, by assumption, the edges $t$ is not adjacent to the loop. If all the terms of the element $e_t\partial e_{ijk}\in F^1_3$ appear in $F^2_3, \dots, F^8_3$, then, also in this case, $t,i,j,k$ have to appear in the same $\pm 1K_3^{(1)}\setminus \mathtt{e}_{23}(-1)$ or in the same $\pm 1K_3^{(1)}$, but this is impossible by construction.
Suppose that $t$ is not a loop and there are two loops in the set $\{i,j,k\}$. In this case $i,j,k$ are the edges of a $1K_2^\circ$. If all the terms of the element $e_t\partial e_{ijk}\in F^1_3$ appear in $F^2_3, \dots, F^8_3$, then $t,i,j,k$ have to appear in the same  $1K_3^\circ$, but this is impossible by construction.

Finally, assume that the edge $t$ is adjacent to all the edges $i,j,k$. Since the underlying graph has at most triple edges and no loops adjacent to a theta graph with only three edges, then in this situation, there are just two cases we should consider.
Suppose that in the set $\{t,i,j,k\}$ there is no loop. If all the terms of the element $e_t\partial e_{ijk}\in F^1_3$ appear in $F^2_3, \dots, F^8_3$, then $t,i,j,k$ have to appear in the same $\pm 1K_3$ or in the same $\mathbb{Z}_3K_3\setminus \mathtt{e}$ or in the same $\mathbb{Z}_3K_3$, but this is impossible by construction.
Suppose that $t$ is not a loop and there is one loop in the set $\{i,j,k\}$. In this case $i,j,k$ are the edges of a $\pm 1K_2^{(1)}$. If all the terms of the element $e_t\partial e_{ijk}\in F^1_3$ appear in $F^2_3, \dots, F^8_3$, then $t,i,j,k$ have to appear in the same $\pm 1K_3^{(1)}\setminus \mathtt{e}_{23}(-1)$ or in the same $\pm 1K_3^{(1)}$, but this is impossible by construction. 

%Therefore, for any element $e_t\partial e_{ijk}\in F^1_3$, at least one of the terms $e_{tjk}, e_{tik}, e_{tij}$ appears only in the expression of $e_t\partial e_{ijk}\in F^1_3$. 
This shows that if $e_t\partial e_{ijk}\in F^p_3$, then can not be expressed linearly by the elements of $\bigcup_{\substack{1\le i \le 8\\ i\ne p}}(F^i_3)$. Since clearly
$$\spann(F_3)=\sum_{i=1}^8\spann(F^i_3)$$
this concludes the proof.
\end{proof}
 In the following example we compute the dimension of $\spann(F_3)$ for the arrangement $\A(\pm 1K_3^{(1)}\setminus \mathtt{e}_{23}(-1))$ associated to the gain graph $\G$ in Figure \ref{Fig:disgraphs2}(a), where  $\langle\G\rangle=\pm 1K_3^{(1)}\setminus \mathtt{e}_{23}(-1)$.
\begin{example}\label{ex:G0comput} Let $\langle\G\rangle=\pm 1K_3^{(1)}\setminus e_{23}(-1)$. Consider the hyperplane arrangement $\A(\G)=\{H_1, \dots, H_6\},$ where $H_1,\dots,H_6$ correspond to the edges $\mathtt{e}_{21}(-1), \mathtt{e}_{12}(1), \mathtt{e}_{13}(1),$ $\mathtt{e}_{13}(-1), e_{23}(1),$ and $\mathtt{e}_{11}(-1)$, respectively. Then the list of $3$-circuits $S$ is $\{126,145, 235, 346\}$. 
Then the number of the elements in $F_3$ is $12$, listed as follows:
$$ e_3 \partial e_{126}=-e_{236}+e_{136}+e_{123}, e_4 \partial e_{126}=-e_{246}+e_{146}+e_{124},$$
$$e_5 \partial e_{126}=-e_{256}+e_{156}+e_{125}, e_2 \partial e_{145}=e_{245}+e_{125}-e_{124},$$
$$e_3 \partial e_{145}=e_{345}+e_{135}-e_{134}, e_6 \partial e_{145}=e_{456}-e_{156}+e_{146},$$
$$e_1 \partial e_{346}=e_{146}-e_{136}+e_{134}, e_2 \partial e_{346}=e_{246}-e_{236}-e_{234},$$
$$e_5 \partial e_{346}=-e_{456}+e_{356}+e_{345}, e_1 \partial e_{235}=e_{135}-e_{125}+e_{123},$$
$$e_4 \partial e_{235}=-e_{345}+e_{245}+e_{234}, e_6 \partial e_{235}=e_{356}-e_{256}+e_{236}.$$
Then an easy computation shows that in this case $\dim(\spann(F_3))=10$. 
\end{example}
In this next example, we compute the dimension of $\spann(F_3)$ for the arrangement $\A(\mathbb{Z}_3K_3\setminus \mathtt{e})$ associated to the gain graph $\G$ in Figure \ref{Fig:disgraphs2}(c), where $\langle\G\rangle=\mathbb{Z}_3K_3\setminus \mathtt{e}$.
\begin{example}\label{ex:G_1comput} 
Let $\langle\G\rangle=\mathbb{Z}_3K_3\setminus \mathtt{e}$. Then the number $3$-circuits $S$ is $7$. 
This implies that the number of the elements in $F_3$ is $35$, and they are all the elements of the form
$$e_t\partial e_{ijk}=e_{tjk}-e_{tik}+e_{tij},$$
for each $3$-circuit $ijk$ and $t\notin\{i,j,k\}$.
Then a direct computation shows that in this case $\dim(\spann(F_3))=34$. 
\end{example}

\begin{remark}\label{rem:dimf3all} Similarly to the previous examples, we can directly compute $\dim(\spann(F_3))$ for all the distinguished gain graphs of Section \ref{sect:listgaingraphimportant}. In particular, if we consider the graphs $\pm 1K_3, 1K_4$ and $1K_3^\circ$, then $\dim(\spann(F_3))=10$. If we consider $\pm 1K_3^{(1)}$, then $\dim(\spann(F_3))=19$. Finally, if we consider $\mathbb{Z}_3K_3$, then $\dim(\spann(F_3))=52$. 
\end{remark}

\begin{lemma} We have the following equalities
\begin{itemize}
\item[] $\dim(\spann(F_3^2))=10k_4$, 
\item[] $\dim(\spann(F_3^3))=10k_3^{\pm}$, 
\item[] $\dim(\spann(F_3^4))=10g_0$, 
\item[] $\dim(\spann(F_3^5))=19k_3^{(1)}$, 
\item[] $\dim(\spann(F_3^6))=10k_3^\circ$, 
\item[] $\dim(\spann(F_3^7))=34g_1$, 
\item[] $\dim(\spann(F_3^8))=52g_2$.
\end{itemize}
\end{lemma}
\begin{proof} Assume that in the graph $\langle\G\rangle$ there are exactly $g_1=p$ distinct subgraphs isomorphic to a $\mathbb{Z}_3K_3\setminus \mathtt{e}$, $\G_1,\dots, \G_p$, none of which is a subgraph of a graph isomorphic to $\mathbb{Z}_3K_3$.  Consider
$$F^7_{3,i}:=\{e_t\partial e_{ijk}~|~e_{ijk}\in C'_3,t\in[n]\setminus\{i,j,k\}, i,j,k \in \G_i\}.$$
Since four edges in the underlying graph of $\langle\G\rangle$ can not appear in two distinct $\mathbb{Z}_3K_3\setminus \mathtt{e}$ at the same time, then none of the terms of the element $e_t\partial e_{ijk}\in F^7_{3,i}$ appear in the elements of $F_3^7\setminus F^7_{3,i}$. This shows that
$$\spann(F^7_3)=\bigoplus_{i=1}^p \spann(F^7_{3,i}).$$
By Proposition \ref{prop:biaseqsamefalkinv}, we have that $\dim(\spann(F^7_{3,i}))=34$ for all $i=1,\dots, p.$ This then implies that
$$\dim(\spann(F^7_3))=\sum_{i=1}^p \dim(\spann(F^7_{3,i}))=34g_1.$$

Using Remark \ref{rem:dimf3all}, the same exact argument used in this case will prove the other equalities.
\end{proof}

\begin{lemma}\label{lemm:dimI32} For an arrangement associated to a gain graph $\G$ without $U_{2,4}$ subgraphs,
% loops adjacent to a theta graph with only three edges, with at most triple parallel edges, and such that $\langle\G\rangle$ does not contain a subgraph isomorphic to $\pm 1K_2^\circ$, 
 we have %For an arrangement associated to a gain graph $\G$ that do not have a subgraph isomorphic to $\pm 1K_2^\circ$, it has no loops adjacent to a theta graph with only three edges and it has at most triple parallel edges, we have
$$\dim((I_2)^3)=(n-2)(k_3+k_2^{(1)}+k_3^\circ+\Theta)-2k_4-2k_3^{\pm}-2g_0-2k_3^\circ-5k_3^{(1)}-g_1-2g_2.$$
\end{lemma}
\begin{proof} By the previous lemmas
$$\dim(\spann(F_3))=\sum_{i=1}^8 \dim(\spann(F^i_3))=$$
$$=[(n-3)(k_3+k_2^{(1)}+k_3^\circ+\Theta)-12k_4-12k_3^{\pm}-12g_0-12k_3^\circ-24k_3^{(1)}-35g_1-54g_2]+ $$
$$ +10k_4+10k_3^{\pm}+10g_0+10k_3^\circ+19k_3^{(1)}+34g_1+52g_2=$$
$$ (n-3)(k_3+k_2^{(1)}+k_3^\circ+\Theta)-2k_4-2k_3^{\pm}-2g_0-2k_3^\circ-5k_3^{(1)}-g_1-2g_2.$$
The thesis follows from the equality
$$\dim((I_2)^3) = k_3+k_2^{(1)}+k_2^\circ+\Theta+\dim(\spann(F_3)).$$
\end{proof}

\begin{proof}[Proof of Theorem \ref{theo:ourmain}]
By Remark \ref{rem:falkinvariantreduct} and Lemma \ref{lem:dimA2} we have
$$\phi_3=2\binom{n+1}{3}-n(\binom{n}{2}-k_3-k_2^{(1)}-k_2^\circ-\Theta)+\binom{n}{3}-\dim((I_2)^3).$$
Because $2\binom{n+1}{3}-n\binom{n}{2} +\binom{n}{3}=0$, then from Lemma \ref{lemm:dimI32} we obtain
$$\phi_3=2(k_3+k_4+k_3^{\pm}+k_2^{(1)}+k_2^\circ+k_3^\circ+g_0+g_2+\Theta)+5k_3^{(1)}+g_1.$$
\end{proof}

Let us see how our formula works on a non-trivial example.
\begin{example}\label{ex5} We want to compute $\phi_3$ for the arrangement associated to the gain graph $\G$ of Figure \ref{fig:finalexample}(a). In order to not create any confusion, in Figure \ref{fig:finalexample}(b) we labeled
%depicted the graph $|\G|$ labeling 
each edge with a letter.

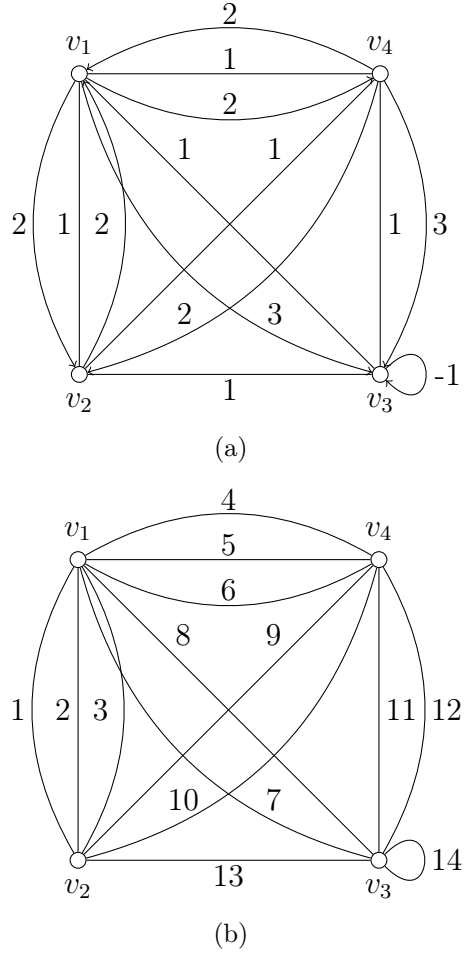
\begin{figure}[htbp]
\centering
\subfigure[]
{
\begin{tikzpicture}[baseline=10pt]
\draw (0,4) node[v, label=above:{$v_1$}](1){};
\draw (0,0) node[v,label=below:{$v_2$}](2){};
\draw (4,0) node[v,label=below:{$v_3$}](3){};
\draw (4,4) node[v,label=above:{$v_4$}](4){};
\draw[] (1)--(2);
\draw[] (1)--(3);
\draw[] (1)--(4);
\draw[] (2)--(3);
\draw[] (2)--(4);
\draw[] (3)--(4);
\draw[bend right, ->>] (1) to (2);
\draw[bend left,<<-] (1) to (2);
\draw[bend right,->>] (1) to (4);
\draw[bend left,<<-] (1) to (4);
\draw[bend right,<-] (2) to (4);
\draw[bend right,->>] (1) to (3);
\draw[bend right,<<-] (3) to (4);
\draw[scale=2,<<-] (3)  to[in=45,out=315,loop] (3);
\draw (-0.8,2) node {2};
\draw (-0.2,2) node {1};
\draw (0.3,2) node {2};
\draw (4.8,2) node {3};
\draw (4.2,2) node {1};
\draw (4.9,0) node {-1};
\draw (2,4.8) node {2};
\draw (2,4.2) node {1};
\draw (2,3.6) node {2};
\draw (2,-0.2) node {1};
\draw (1.4,3) node {1};
\draw (2.6,3) node {1};
\draw (1.4,0.8) node {2};
\draw (2.6,0.8) node {3};
\end{tikzpicture}
}
\hspace{14mm}
\subfigure[]
{
\begin{tikzpicture}[baseline=10pt]
\draw (0,4) node[v, label=above:{$v_1$}](1){};
\draw (0,0) node[v,label=below:{$v_2$}](2){};
\draw (4,0) node[v,label=below:{$v_3$}](3){};
\draw (4,4) node[v,label=above:{$v_4$}](4){};
\draw[] (1)--(2);
\draw[] (1)--(3);
\draw[] (1)--(4);
\draw[] (2)--(3);
\draw[] (2)--(4);
\draw[] (3)--(4);
\draw[bend right] (1) to (2);
\draw[bend left] (1) to (2);
\draw[bend right] (1) to (4);
\draw[bend left] (1) to (4);
\draw[bend right] (2) to (4);
\draw[bend right] (1) to (3);
\draw[bend right] (3) to (4);
\draw[scale=2] (3)  to[in=45,out=315,loop] (3);
\draw (-0.8,2) node {a};
\draw (-0.2,2.05) node {b};
\draw (0.3,2) node {c};
\draw (4.9,2) node {l};
\draw (4.3,2) node {k};
\draw (4.9,0) node {n};
\draw (2,4.8) node {d};
\draw (2,4.2) node {e};
\draw (2,3.6) node {f};
\draw (2,-0.2) node {m};
\draw (1.4,3) node {h};
\draw (2.6,3) node {i};
\draw (1.4,0.8) node {j};
\draw (2.6,0.8) node {g};
\end{tikzpicture}
}
\caption{The gain graph $\G$ and its underlying graph.}
\label{fig:finalexample}
\end{figure}
In order to compute $\phi_3$ with the formula \eqref{eq:ourmainformula}, we need to compute the following:
\begin{itemize}
\item[] $k_3=|\{\{\text{b, e, i}\},\{\text{b, h, m}\},\{\text{e, h, k}\},\{\text{i, k, m}\},\{\text{a, f, i}\},\{\text{a, e, j}\},$ $\{\text{b, d, j}\},$ $\{\text{c, d, i}\},\{\text{e, g, l}\}\}|=9;$
\item[] $k_4=|\{\{\text{b, e, h, i, k, m}\}\}|=1;$
\item[] $k_3^{\pm}=0;$
\item[] $k_2^{(1)}=|\{\{\text{g, h, n}\},\{\text{k, l, m}\}\}|=2;$
\item[] $k_2^\circ=0;$
\item[] $k_3^\circ=0;$
\item[] $g_0=|\{\{\text{e, g, h, k, l, n}\}\}|=1;$
\item[] $g_2=0;$
\item[] $\Theta=|\{\{\text{a, b, c}\},\{\text{d, e, f}\}\}|=2;$
\item[] $k_3^{(1)}=0;$
\item[] $g_1=|\{\{\text{a, b, c, d, e, f, i, j}\}\}|=1.$
\end{itemize} 
From formula \eqref{eq:ourmainformula}, we obtain $$\phi_3=2(9+1+0+2+0+0+1+0+2)+0+1=31.$$ Notice that if we would try to compute the dimension of $F_3$ directly, we would have to write $143$ equations in the $e_{ijk}$.
\end{example}

\section{Matroidal interpretation}

In this section, we will give a matroidal interpretation of our main theorem. In Theorem \ref{theo:ourmain}, the formula \eqref{eq:ourmainformula} of $\phi_3$ is expressed in terms of the numbers of subgraphs of the given gain graph. 
Since some of the subgraphs appearing in the formula \eqref{eq:ourmainformula} describe different realizations of the same matroid, we are able to give a new formula just in terms of the numbers of the submatroids of the frame matroid of the given gain graph. In this way, we obtain a simpler and more compact formula.

In Theorem \ref{theo:ourmain}, we consider the class of arrangements associated to gain graphs $\G$ such that $\langle\G\rangle$ does not have a subgraph isomorphic to $\pm 1K_2^\circ$, it has no loops adjacent to a theta graph with only three edges and it has at most triple parallel edges. This class coincides with the class of arrangements associated to gain graphs such that the underlying matroid has no rank-two flats of size greater than three.

From the list of gain graphs in Section \ref{sect:listgaingraphimportant}, it is immediate to prove the following.
\begin{lemma}\label{lemma:shpeequivmatr} The graphs $1K_3$, $1K_2^{\circ}$, $\pm 1K_2^{(1)}$ and the contrabalanced theta graph with three edges have isomorphic underlying matroids. Similarly, the graphs $1K_4$, $\pm 1K_3$, $1K_3^{\circ}$ and $\pm 1K_3^{(1)}\setminus \mathtt{e}_{23}(-1)$ have isomorphic underlying matroids.
\end{lemma}
\begin{definition} We denote by $M(K_3)$ the frame matroid associated to $1K_3$, and by $M(K_4)$ the one associated to $1K_4$.
\end{definition}

%From list of gain graphs in section \ref{sect:listgaingraphimportant} we can find that there is only one circuit in the graph $[1K_3],$ $[1K_2^{\circ}]$ and $\pm 1K_2,$ and each of them contains three edges. Therefore, their graphic matroids can be regarded as the same one $M[K_3].$ 
%
%In the similar way, each of the graphs $[1K_4],$ $\pm 1K_3,$ $[1K_3^{\circ}]$ and $\pm 1K_3^{(1)}\setminus e_{23}(-1)$ has six edges with four circuits, and their matroids are isomorphic, denoted by $M[K_4].$

Let $\A(\G)$ be an arrangement associated to the gain graph $\G$. Denote by $k_l$ the number of submatroid of $M(\G)$ isomorphic to $M(K_l)$, by $g_1$ the number of submatroid of $M(\G)$ isomorphic to $M(\mathbb{Z}_3K_3\setminus \mathtt{e})$, by $g_2$ the number of submatroid of $M(\G)$ isomorphic to $M(\mathbb{Z}_3K_3)$, and by $k_3^{(1)}$ the number of submatroid of $M(\G)$ isomorphic to $M(\pm 1K_3^{(1)})$.

%Given a gain graph $\G$, let $\overline{k}_3$ denote the number of subgraphs of $\G$ whose corresponding graphic matroid is isomorphic to $M[K_3],$ and $\overline{k}_4$ denote the number of subgraphs of $\G$ whose corresponding graphic matroid is isomorphic to $M[K_4].$ The numbers $\mathbb{Z}_3K_3,$ $\theta,$ $g_1$ and $k_3^{(1)}$ are same as in the main theorem \ref{theo:ourmain}. Then we can rewrite our main theorem as following.

\begin{remark}\label{rem:d31phi3}
Consider the gain graph $\G=\pm 1K_3^{(1)}$ described in Figure \ref{Fig:disgraphs2}(b). $M(\G)$ is the well-known non-Fano matroid. $\langle\G\rangle$ has two subgraphs isomorphic to $\pm 1K_2^{(1)}$ and four isomorphic to $1K_3$, and hence, $M(\G)$ has six submatroid isomorphic to $M(K_3)$. Similarly, $\langle\G\rangle$ has one subgraph isomorphic to $\pm 1K_3$ and two to $\pm 1K_3^{(1)}\setminus \mathtt{e}_{23}(-1)$, and hence, $M(\G)$ has three submatroid isomorphic to $M(K_4)$.
By Theorem \ref{theo:ourmain}, $\phi_3(\A(\G))=17$, that it coincides with $2(k_3+k_4)-1$.
\end{remark}

We are now able to rewrite formula \eqref{eq:ourmainformula} of Theorem \ref{theo:ourmain} just in terms of submatroid.

\begin{theorem}\label{newtheo:ourmain}
For an arrangement associated to a gain graph $\G$ such that the underlying matroid has no rank-two flats of size greater than three, we have
\begin{equation}\label{eq:ournewmainformula}
\phi_3=2(k_3+k_4+g_2)-k_3^{(1)}+g_1,
\end{equation}
\end{theorem}

\begin{proof}
%The submotroid  $M([1K_3]^\lambda),$ $M([1K_2^{\circ}]^\lambda)$ and $M([\pm 1K_2]^\lambda)$ are isomorphic to $M(K_3),$ $k_3$ denotes the total number of these matroids. And $k_4$ denote the total number of submatroids of $M([1K_4]^\lambda),$ $M([1K_3^\circ]^\lambda),$ $M([\pm 1K_3]^\lambda)$ and $M([\pm 1K_3^{(1)}\setminus e_{23}(-1)]^\lambda)$ which are isomorphic to $M(K_4).$ But one $M([\pm 1K_2]^\lambda)$ has one submatroid isomorphic to  $M([\pm 1K_3]^\lambda)$ and two to $M([\pm 1K_3^{(1)}\setminus e_{23}(-1)]^\lambda)$, so if the gain graph $\G$ has the subgraph $\pm 1K_3^{(1)},$ we will have repeatly counted the number of submatroids isomorphic to  $M([\pm 1K_3]^\lambda)$ and to $M([\pm 1K_3^{(1)}\setminus e_{23}(-1)]^\lambda)$ when we do not change the ccoefficient of $k_3^{(1)}.$  According to the formula \eqref{eq:ourmainformula}, we can get the new formula \eqref{eq:ournewmainformula}.
From the formula \eqref{eq:ourmainformula}, we can get the new formula \eqref{eq:ournewmainformula}, by the use of Lemma~\ref{lemma:shpeequivmatr} and Remark~\ref{rem:d31phi3}.
\end{proof}

In the following example we will use the new formula \eqref{eq:ournewmainformula} to compute $\phi_3$ for the gain graph $\G$ of figure \ref{fig:finalexample} in the example \ref{ex5}.
\begin{example}
%Use the new formula to copute $\phi_3$ for the gain graph $\G$ of figure \ref{fig:finalexample} in the example \ref{ex5}. 
There are $13$ submatroids of $M(\G)$ isomorphic to $M(K_3)$ ($9$ submotroids come from the subgraphs isomorphic to $1K_3$ and $2$ from the subgraphs isomorphic to $\pm 1K_2^{(1)}$, and $2$ from contrabalanced theta graph with only three edges). There are also $2$ submatroids of $M(\G)$ isomorphic to $M(K_4)$ (one is isomorphic to $1K_4$ and the other one to $\pm 1K_3^{(1)}\setminus \mathtt{e}_{23}(-1)$). Besides these submatroids, there is $1$ submatroid isomorphic to $M(\mathbb{Z}_3K_3\setminus \mathtt{e})$.

Therefore, we obtain
$$\phi_3=2(13+2+0)-0+1=31.$$
This coincides with the result of the computation in Example \ref{ex5}.

\end{example}

\paragraph{\textbf{Acknowledgements}} The authors would like to thank the referees for their valuable comments which helped to improve the manuscript.

%\bibliography{bibliothesis}{}
%\bibliographystyle{plain}

\end{document}